\documentclass[12pt]{amsart}
\usepackage{amssymb,amsmath,amsfonts,latexsym,amscd}
\usepackage{bm}
\newcommand{\be}{\begin{equation}}
\newcommand{\ee}{\end{equation}}
\newcommand{\bea}{\begin{eqnarray}}
\newcommand{\eea}{\end{eqnarray}}
\newcommand{\bean}{\begin{eqnarray*}}
\newcommand{\eean}{\end{eqnarray*}}
\newcommand{\brray}{\begin{array}}
\newcommand{\erray}{\end{array}}
\newcommand{\ben}{\begin{equation}{nonumber}}
\newcommand{\een}{\end{equation}{nonumber}}
\newcommand{\newsection}[1]{\setcounter{equation}{0} \setcounter{dfn}{0}
\section{#1}}
\newtheorem{dfn}{Definition}[section]
\newtheorem{thm}[dfn]{Theorem}
\newtheorem{lmma}[dfn]{Lemma}

\newtheorem{ppsn}[dfn]{Proposition}
\newtheorem{crlre}[dfn]{Corollary}
\newtheorem{xmpl}[dfn]{Example}
\newtheorem{rmrk}[dfn]{Remark}
\newcommand{\bdfn}{\begin{dfn}}
\newcommand{\bthm}{\begin{thm}}

\numberwithin{equation}{section}

\newcommand{\blr}{\begin{list}{$($\roman{cnt1}$)$} {\usecounter{cnt1}
        \setlength{\topsep}{0pt} \setlength{\itemsep}{0pt}}}
\newcommand{\bla}{\begin{list}{$($\alph{cnt2}$)$} {\usecounter{cnt2}
       \setlength{\topsep}{0pt} \setlength{\itemsep}{0pt}}}
\newcommand{\bln}{\begin{list}{$($\arabic{cnt3}$)$} {\usecounter{cnt3}
                \setlength{\topsep}{0pt} \setlength{\itemsep}{0pt}}}
\newcommand{\el}{\end{list}}

\newcommand{\blmma}{\begin{lmma}}
\newcommand{\bppsn}{\begin{ppsn}}
\newcommand{\bcrlre}{\begin{crlre}}
\newcommand{\bxmpl}{\begin{xmpl}}
\newcommand{\brmrk}{\begin{rmrk}}
\newcommand{\edfn}{\end{dfn}}
\newcommand{\ethm}{\end{thm}}
\newcommand{\elmma}{\end{lmma}}

\newcommand{\eppsn}{\end{ppsn}}
\newcommand{\ecrlre}{\end{crlre}}
\newcommand{\exmpl}{\end{xmpl}}
\newcommand{\ermrk}{\end{rmrk}}

\def\a*{{\cal A}_{h,*}}
\def\B{{\cal B}(h)}
\def\B1{{\cal B}_1(h)}
\def\b{{\cal B}^{\rm s.a.}(h)}
\def\b1{{\cal B}^{\rm s.a.}_1(h)}

\newcommand {\im}{\textup{i}}

\newcommand {\Exp} {\textup{e}}

\newcommand {\Tr}{\textup {Tr}}

\begin{document}

 \title{Helton-Howe-Carey-Pincus Trace Formula and Krein's Theorem}

\author[Chattopadhyay] {Arup Chattopadhyay $^{^{1)}}$ }
\address{
1) A. Chattopadhyay: Indian Institute of Technology Guwahati\\ Department of Mathematics
\\ Guwahati- 560059 \\Kamrup, Assam, India.}

\email{2003arupchattopadhyay@gmail.com, arupchatt@iitg.ernet.in}

\author[Sinha]{Kalyan B. Sinha $^{^{2)}}$}

\address{
2) K. B. Sinha: J.N.Centre for Advanced Scientific Research\\ and Indian Institute of Science,\\ 
 Bangalore\\ India.}

 \email{kbs@jncasr.ac.in}

\begin{abstract}
 In this article, we derive the 
 Helton-Howe-Carey-Pincus trace formula as a consequence of Krein's trace formula.
 \end{abstract}

\maketitle

\vspace{0.1in}


Spectral integral.

%
%
%
%
%
%

\newsection{Introduction}\label{sec:1}

\noindent \textbf{Notation:} In the following, we shall use the notations given  below:
$\mathcal{H}$, $\mathcal{B}(\mathcal{H})$, $\mathcal{B}_{sa}(\mathcal{H})$,
$\mathcal{B}_1(\mathcal{H})$, $\mathcal{B}_{1+}(\mathcal{H})$, $\mathcal{B}_{1-}(\mathcal{H})$,
$\mathcal{B}_2(\mathcal{H})$,
$\mathcal{B}_p(\mathcal{H})$
denote a separable Hilbert space, set of bounded linear operators, set of bounded self-adjoint linear operators, 
set of trace class operators, set of positive
trace class operators, set of negative
trace class operators, set of Hilbert-Schmidt operators and Schatten-p class operators respectively
with $\|.\|_p$ as the associated Schatten-p norm. Furthermore by $\sigma(A)$, $E_A(\lambda)$, $D(A)$,
$\rho(A)$, we shall mean spectrum, spectral family, domain, resolvent set,
and resolvent of a self-adjoint operator $A$ respectively,
and $\textup{Tr}(A)$ will denote the trace of a trace class operator $A$ in $\mathcal{H}$.
Also we denote 
the set of natural numbers and the set of real numbers by $\mathbb{N}$ and $\mathbb{R}$ 
respectively. The set  $C(I)$ is
the Banach space of continuous
functions over a compact interval $I\subseteq \mathbb{R}$ with sup-norm $\|.\|_{\infty}$, 
and  $C^n(I)~(n\in \mathbb{N} \cup \{0\})$, the space of $n$-th continuously differentiable
functions over a compact interval $I$ with norm 
\[
 \|f\|^{n} = \sum_{j=0}^n\|f^{(j)}\|_{\infty} ~~~\textup{for}~~f\in C^n(I)
\]
and $f^{(j)}$ is the $j$-th derivative of $f$ (for $n=0$, $C^n(I)$ is $C(I)$),
and $L^p(\mathbb{R})$ the standard Lebesgue space. We shall denote
$f^{(1)}$, the first derivative, as $f'$.
Next we define the class $C_1^1(I)\subseteq C(I)$
as follows
\[
 C_1^1(I) = \Big\{f\in C(I): \|f\|_1^{1}=
 \frac{1}{\sqrt{2\pi}}\int_{\mathbb{R}}
 |\hat{f}(\alpha)|(1+|\alpha|)d\alpha<\infty\Big\},
\]
where $\hat{f}$ is the Fourier transform of $f$;
and it is easy to see that $C_1^1(I)\subseteq C^1(I)$ (since $\|.\|^{1}\leq \|.\|_1^{1}$); 
and denote the set all $f\in C_1^1(I)$ such that 
$f'\geq 0$ ($\leq 0$ respectively) by $C_{1+}^1(I)$ ($C_{1-}^1(I)$ respectively). Similarly
we denote the set of all polynomials with complex coefficients in $I$ by $\mathcal{P}(I)$
and the set of all $p\in \mathcal{P}(I)$ such that $p'\geq 0$ ($\leq 0$ respectively)
by $\mathcal{P}_{+}(I)$ ($\mathcal{P}_{-}(I)$ respectively).

Let $T\in \mathcal{B}(\mathcal{H})$ be a hyponormal operator, 
that is, $[T^*,T]\geq 0$. Set $T=X+\im Y$, where $X,Y\in \mathcal{B}_{sa}(\mathcal{H})$
and it is known that $\textup{Re}(\sigma(T))=\sigma(X)$, $\textup{Im}(\sigma(T))=\sigma(Y)$
\cite{putnam}, $[T^*,T]\in \mathcal{B}_1(\mathcal{H})$
if an additional assumption of finiteness
of spectral multiplicity is assumed \cite{bergershaw, kato,  martinputinar, putnam}.
A hyponormal operator $T$ is said to be
purely hyponormal if there exists no subspace $\mathcal{S}$ of $\mathcal{H}$ which is invariant
under $T$ such that the restriction of $T$ to $\mathcal{S}$ is normal. For a purely hyponormal
operator $T,$ it is also known that its real and imaginary parts, that is, $X$ and $Y$ are spectrally
absolutely continuous \cite{kato, putnam3}.

\textbf{\emph{(A):~The main assumption of the whole paper is 
that}} $T=X+\textup{i}Y$ \textbf{\emph{is a purely hyponormal operator in 
$\mathcal{B}(\mathcal{H})$ such that}} 
$[T^*,T]=-2\textup{i}[Y,X]= 2D^2\in \mathcal{B}_{1+}(\mathcal{H})$\textbf{\emph{
and $\sigma(X) \cup \sigma(Y)\subseteq [a,b]$
$\subseteq \mathbb{R}$.}}

\newsection{Main Section}
Let us start with few lemmas which will be useful to prove our main result. 
\begin{lmma}\label{lma1}
Let $T$ satisfy $\textbf{(A)}$.
Then for $\psi\in C_1^1([a,b])$ or $\mathcal{P}([a,b])$,
\begin{equation}
[\psi(Y),X]\in \mathcal{B}_1(\mathcal{H})
\quad \textup{and} \quad
-\im \textup{Tr}\{[\psi(Y),X]\} = \textup{Tr}\{\psi'(Y)D^2\}.
\end{equation}
Similarly, 
\begin{equation}\label{useq35}
[Y, \psi(X)]\in \mathcal{B}_1(\mathcal{H})
\quad \textup{and} \quad
-\im \textup{Tr}\{[Y,\psi(X)]\} = \textup{Tr}\{\psi'(X)D^2\}.
\end{equation}
\end{lmma}
\begin{proof}
Now for $\psi\in C_1^1([a,b])$ we have
\begin{equation}\label{useq80}
 \begin{split}
  &[\psi(Y),X] = \frac{1}{\sqrt{2\pi}} \int_{\mathbb{R}}\hat{\psi}(\alpha)[\Exp^{\im \alpha Y},X]d\alpha
  = \frac{1}{\sqrt{2\pi}} 
  \int_{\mathbb{R}}\hat{\psi}(\alpha)\left(\Exp^{\im \alpha Y}X-X\Exp^{\im \alpha Y}\right)d\alpha\\
 & = \frac{1}{\sqrt{2\pi}} \int_{\mathbb{R}}\hat{\psi}(\alpha)d\alpha \int_0^{\alpha} 
 \im \Exp^{\im (\alpha-\beta)Y}[Y,X]\Exp^{\im \beta Y} d\beta= - \frac{1}{\sqrt{2\pi}}
 \int_{\mathbb{R}}\hat{\psi}(\alpha)d\alpha \int_0^{\alpha} 
 \Exp^{\im (\alpha-\beta)Y}D^2\Exp^{\im \beta Y}d\beta.
 \end{split}
\end{equation}
Since $D^2\in \mathcal{B}_1(\mathcal{H})$ and 
$\int_{\mathbb{R}}|\hat{\psi}(\alpha)||\alpha|d\alpha < \infty,$ then
from the above equation \eqref{useq80} we conclude that
\begin{equation*}
 [\psi(Y),X]\in \mathcal{B}_1(\mathcal{H})~~~\text{and}~~~
 \|[\psi(Y),X]\|_1\leq  \|D^2\|_1 \frac{1}{\sqrt{2\pi}} \int_{\mathbb{R}} |\hat{\psi}(\alpha)||\alpha|d\alpha.
\end{equation*}
Moreover,
\begin{equation*}
\begin{split}
& -\im \Tr\{[\psi(Y),X]\} 
 =  \frac{\im}{\sqrt{2\pi}}
 \int_{\mathbb{R}}\hat{\psi}(\alpha)d\alpha \int_0^{\alpha} 
 \Tr\{\Exp^{\im \alpha Y}D^2\}d\beta = \frac{\im}{\sqrt{2\pi}}
 \int_{\mathbb{R}}\alpha \hat{\psi}(\alpha)d\alpha
 \Tr\{\Exp^{\im \alpha Y}D^2\}\\
 & =   \Tr\{\frac{1}{\sqrt{2\pi}} \int_{\mathbb{R}}\im \alpha \hat{\psi}(\alpha)
 \Exp^{\im \alpha Y} d\alpha ~~D^2\} =  \Tr\{\psi'(Y)D^2\},
\end{split}
\end{equation*}
where we have used the cyclicity of trace and the fact that
\begin{equation*}
 \psi'(\beta) = \frac{1}{\sqrt{2\pi}}\int_{\mathbb{R}}\im \alpha \hat{\psi}(\alpha)
 \Exp^{\im \alpha \beta} d\alpha. 
\end{equation*}
Next for $\psi(t)=\sum\limits_{j=0}^nc_jt^j\in \mathcal{P}([a,b])$ we get
\begin{equation}\label{useq95}
 \begin{split}
  [\psi(Y),X] = \sum_{j=0}^nc_j[Y^j,X]
  = \sum_{j=0}^nc_j \sum_{k=0}^{j-1}Y^{j-k-1}[Y,X]Y^j
  = \im  \sum_{j=0}^nc_j \sum_{k=0}^{j-1}Y^{j-k-1}D^2 Y^k.
 \end{split}
\end{equation}
Since $D^2\in \mathcal{B}_1(\mathcal{H}),$ then
from the above equation \eqref{useq95} we conclude that
\begin{equation*}
 [\psi(Y),X]\in \mathcal{B}_1(\mathcal{H})
 \quad \text{and} \quad \left\|[\psi(Y),X]\right\|_1
 \leq \|D^2\|_1\left(\sum_{j=0}^n j |c_j| \|Y\|^{j-1}\right).
\end{equation*}
Furthermore,
\begin{equation*}
 \begin{split}
& -\im \Tr\{[\psi(Y),X]\}
  =  \sum_{j=0}^nc_j \sum_{k=0}^{j-1}\Tr\{Y^{j-k-1}D^2 Y^k\}
  =  \sum_{j=0}^nc_j \sum_{k=0}^{j-1}\Tr\{Y^{j-1}D^2 \}\\
& =  \sum_{j=0}^nj c_j \Tr\{Y^{j-1}D^2 \}
=  \Tr\{\sum_{j=0}^nj c_jY^{j-1}D^2\}=  \Tr\{\psi'(Y)D^2\},
 \end{split}
\end{equation*}
where we have used the cyclicity of property of trace.
By interchanging the role of $X$ and $Y$ in the above calculations, we conclude that
\begin{equation*}
[Y, \psi(X)]\in \mathcal{B}_1(\mathcal{H})
\quad \textup{and} \quad
-\im \textup{Tr}\{[Y,\psi(X)]\} = \textup{Tr}\{\psi'(X)D^2\}.
\end{equation*}

This completes the proof.
\end{proof}

\begin{lmma}\label{lmma2}
Let $T$ satisfy $\textbf{(A)}$. Then
 $-\im[\psi(Y),X]$ $\in \mathcal{B}_{1\pm}(\mathcal{H})$ according as  
 $\psi\in C_{1\pm}^1([a,b])$ or $\mathcal{P}_{\pm}([a,b])$
 respectively. Similarly, $-\im[Y,\psi(X)]$ $\in \mathcal{B}_{1\pm}(\mathcal{H})$
 according as $\psi\in C_{1\pm}^1([a,b])$ or $\mathcal{P}_{\pm}([a,b])$ respectively. 
\end{lmma}
\begin{proof}
Let $\psi\in C_1^1([a,b]).$ Then from equation \eqref{useq80} in lemma~\ref{lma1}, we have
\[
 -\im [\psi(Y),X] = \frac{1}{\sqrt{2\pi}} \int_{\mathbb{R}}\im \hat{\psi}(\alpha)d\alpha \int_0^{\alpha} 
 \Exp^{\im (\alpha-\beta)Y}D^2\Exp^{\im \beta Y}d\beta.
\]
Next by the spectral theorem for $Y$ we get
\begin{equation*}
 \begin{split}
& -\im [\psi(Y),X]
 = \frac{1}{\sqrt{2\pi}} \int_{\mathbb{R}} \im \hat{\psi}(\alpha) d\alpha
\int_0^{\alpha}
d\beta\int_a^b \int_a^b \Exp^{\im(\alpha-\beta) t}\Exp^{\im\beta t'}
E^{(Y)}(dt)D^2E^{(Y)}(dt'),
\end{split}
\end{equation*}
where $E^{(Y)}(.)$ is the spectral
family of the self-adjoint operator $Y$. 
Note that $\mathcal{E}(\Delta \times \delta) (S) \equiv E^{(Y)}(\Delta)SE^{(Y)}(\delta)$
($S\in \mathcal{B}_2(\mathcal{H})$ and $\Delta \times \delta \subseteq 
\mathbb{R}\times \mathbb{R}$) extends to a spectral measure (finite) on $\mathbb{R}^2$
in the Hilbert space $\mathcal{B}_2(\mathcal{H})$. Therefore by Fubini's theorem 
\begin{equation}\label{useq15}
 \begin{split}
&  -\im [\psi(Y),X] = \frac{1}{\sqrt{2\pi}} \int_{\mathbb{R}} \im \hat{\psi}(\alpha) d\alpha
\int_a^b \int_a^b \Exp^{\im\alpha t} \int_0^{\alpha}
d\beta~\Exp^{-\im \beta (t-t')}
E^{(Y)}(dt)D^2E^{(Y)}(dt') \\
& = \frac{1}{\sqrt{2\pi}} \int_{\mathbb{R}} \im \hat{\psi}(\alpha) d\alpha
\int_a^b \int_a^b \frac{\Exp^{\im \alpha t'}-\Exp^{\im\alpha t} }{\im(t'-t)}
E^{(Y)}(dt)D^2E^{(Y)}(dt') \}\\
& = \frac{1}{\sqrt{2\pi}}
\int_a^b \int_a^b
\frac{\psi(t')
-\psi(t) }{t'-t}
E^{(Y)}(dt)D^2E^{(Y)}(dt') \}
= \frac{1}{\sqrt{2\pi}} \int_{[a,b]^2} \tilde{\psi}(t,t')
\mathcal{E}(dt\times dt')(D^2),
\end{split}
\end{equation}
where
\begin{equation*}
\tilde{\psi}(t,t')= 
\begin{cases}
\frac{\psi(t')
-\psi(t) }{t'-t} , &\text{if $t \neq t'$,}\\
\psi'(t), ~&\text{if $t = t'$.}
\end{cases}
\end{equation*}
Note that for $\psi \in C^1_{1\pm}([a,b])$,
we have 
\[
\tilde{\psi}(t,t') \geq 0 \quad \text{(or}~ \leq 0) \quad \text{respectively for} \quad t,t'\in [a,b]
\]
and hence from the equation \eqref{useq15} we conclude that $-\im [\psi(Y),X]\geq 0$ or
$-\im [\psi(Y),X]\leq 0$ accordingly. By repeating the above calculations 
with $X$ and $Y$ interchanged we get that
\begin{equation}\label{useq52}
 -\im [Y,\psi(X)]= \frac{1}{\sqrt{2\pi}} \int_{\mathbb{R}}\im \hat{\psi}(\alpha)d\alpha \int_0^{\alpha} 
 \Exp^{\im (\alpha-\beta)X}D^2\Exp^{\im \beta X}d\beta = \frac{1}{\sqrt{2\pi}} 
 \int_{[a,b]^2}\tilde{\psi}(t,t') \tilde{\mathcal{E}}(dt\times dt')(D^2),
\end{equation}
where $\tilde{\mathcal{E}}(\Delta \times \delta) (S) = E^{(X)}(\Delta)SE^{(X)}(\delta)$
($S\in \mathcal{B}_2(\mathcal{H})$ and $\Delta \times \delta \subseteq 
\mathbb{R}\times \mathbb{R}$) extends to a spectral measure on $\mathbb{R}^2$
in the Hilbert space $\mathcal{B}_2(\mathcal{H})$ and $E^{(X)}(.)$ is the spectral
family of the self-adjoint operator $X$. Therefore as above we conclude that 
$-\im [Y,\psi(X)]\geq 0$ or $\leq 0$ if $\psi \in C^1_{1+}([a,b])$ or $C^1_{1-}([a,b])$
respectively. 

Next assume that $\psi(t)=\sum_{j=0}^nc_jt^j\in \mathcal{P}([a,b])$.
Then from equation \eqref{useq95} in lemma~\ref{lma1}, we get
\begin{equation*}
-\im [\psi(Y),X]
  =  \sum_{j=0}^nc_j \sum_{k=0}^{j-1}Y^{j-k-1}D^2 Y^k.
\end{equation*}
Thus by using spectral theorem for $Y$ and Fubini's theorem we get
\begin{equation}\label{useq51}
 \begin{split}
& -\im [\psi(Y),X]
  =  \sum_{j=0}^nc_j \sum_{k=0}^{j-1}Y^{j-k-1}D^2 Y^k
  = \sum_{j=0}^nc_j \sum_{k=0}^{j-1} \int_a^b \int_a^b t^{j-k-1}t'^k
E^{(Y)}(dt)D^2E^{(Y)}(dt')\\
&  = \sum_{j=0}^nc_j \int_a^b \int_a^b \frac{t'^j-t^j}{t'-t}
E^{(Y)}(dt)D^2E^{(Y)}(dt')
= \int_{[a,b]^2} \tilde{\psi}(t,t')
\mathcal{E}(dt\times dt')(D^2),
 \end{split}
\end{equation}
where $E^{(Y)}(.)$ is the spectral
family of the self-adjoint operator $Y$ and $\tilde{\psi}(t,t')$, $\mathcal{E}(dt\times dt')$
as in equation \eqref{useq15}. Note that 
\[
 \tilde{\psi}(t,t')\geq 0 (~\text{or}~\leq 0) \quad \text{according as} \quad \psi\in \mathcal{P}_+
 (~\text{or}~ \mathcal{P}_-)
\]
and hence from \eqref{useq51} we conclude that 
\[
 -\im [\psi(Y),X] \geq 0 (~\text{or}~\leq 0) \quad \text{according as} \quad \psi\in \mathcal{P}_+
 (~\text{or}~ \mathcal{P}_-).
\]
By repeating the above calculations with $X$ and $Y$ interchanged we get that
\[
 -\im [Y,\psi(X)]= 
 \int_{[a,b]^2}\tilde{\psi}(t,t') \tilde{\mathcal{E}}(dt\times dt')(D^2),
\]
where $\tilde{\mathcal{E}}(dt\times dt')$ as in \eqref{useq52}. Therefore as above we conclude that
\[
 -\im [Y,\psi(X)] \geq 0 (~\text{or}~\leq 0) \quad \text{according as} \quad \psi\in \mathcal{P}_+
 (~\text{or}~ \mathcal{P}_-).
\]
This completes the proof.
\end{proof}

\begin{lmma}\label{lmma6}
Let $T$ satisfy $\textbf{(A)}.$ Then $-\im [\psi(Y),\phi(X)]\in \mathcal{B}_{1\pm}(\mathcal{H})$
according as $\psi$ and $\phi$ $\in C_{1\pm}^1([a,b])$ or $\mathcal{P}_{\pm}([a,b])$ respectively.
\end{lmma}
\begin{proof}
Let  $\psi, \phi \in C_{1}^1([a,b])$ or $\mathcal{P}([a,b])$. Then by repeating the same calculations as in \eqref{useq15} 
and in \eqref{useq51}  of the above lemma~\ref{lmma2} we get,
\begin{equation}\label{useq54}
 \begin{split}
&  -\im [\psi(Y),\phi(X)] 
= \frac{1}{\sqrt{2\pi}}
\int_a^b \int_a^b
\frac{\psi(t')
-\psi(t) }{t'-t}
E^{(Y)}(dt)(-\im[Y,\phi(X)])E^{(Y)}(dt') \}\\
& \hspace{2.95cm} = \frac{1}{\sqrt{2\pi}} \int_{[a,b]^2} \tilde{\psi}(t,t')
\mathcal{E}(dt\times dt')(-\im[Y,\phi(X)]), 
 \end{split}
\end{equation}
where the description of $\tilde{\psi}$ and $\mathcal{E}$ has been discussed in the proof of lemma~\ref{lmma2}.
On the other hand from lemma~\ref{lmma2} we say that
\begin{equation}\label{useq55}
 -\im[Y,\phi(X)] \in \mathcal{B}_{1\pm}(\mathcal{H})
 \quad \text{according as} \quad \phi\in C_{1\pm}^1([a,b])  \quad \text{or} \quad \mathcal{P}_{\pm}([a,b])\quad \text{respectively.} 
\end{equation}
Therefore by combining equations \eqref{useq54} and \eqref{useq55} we conclude that
\[
 -\im[\psi(Y),\phi(X)] \in \mathcal{B}_{1\pm}(\mathcal{H})
 \quad \text{according as} \quad \psi, \phi\in C_{1\pm}^1([a,b]) \quad \text{or} \quad \mathcal{P}_{\pm}([a,b])\quad \text{respectively.} 
\]
This completes the proof.
\end{proof}

As the title suggests, we shall next state Krein's theorem and
study its consequences on commutators like 
$[\psi(Y),\phi(X)]$ for $\psi, \phi$ $\in C_1^1([a,b])$ or $\mathcal{P}([a,b])$.
\begin{ppsn} \label{prop1}(Krein's Theorem)\cite{mgkrein53,krein, sinmoha,Voiculescu87}
Let $H$ and $H_0$ be two bounded self-adjoint operators in 
$\mathcal{H}$ such that $V$ $=H-H_0 \in \mathcal{B}_1(\mathcal{H})$.
Then there exists a unique $\xi_{H_0,H}(.)\in L^1(\mathbb{R})$
such that for $\phi \in C_1^1([a,b])$ or $\mathcal{P}([a,b])$, $\phi(H)-\phi(H_0)\in \mathcal{B}_1(\mathcal{H})$
and 
\[
 \Tr\{\phi(H)-\phi(H_0)\} = \int_a^b \phi'(\lambda)\xi_{H_0,H}(\lambda)d\lambda,
\]
where $\sigma(H) \cup \sigma(H_0)\subseteq [a,b]$. Furthermore 
\[
\int_a^b |\xi_{H_0,H}(\lambda)| d\lambda \leq \|V\|_1 ;
\int_a^b \xi_{H_0,H}(\lambda) d\lambda = \Tr V,
\]
and if 
$V\in \mathcal{B}_{1+}(\mathcal{H})$ or $\mathcal{B}_{1-}(\mathcal{H})$,
then $\xi_{H_0,H}(\lambda)$ is positive or negative respectively 
for almost all $\lambda \in [a,b].$
\end{ppsn}

\begin{thm}\label{thm3}
Assume $\textbf{(A)}$.
Let $\phi$ and $\psi$ be two
complex-valued functions
such that $\phi,$ $\psi \in C_1^1([a,b])$ or $\mathcal{P}([a,b])$.
Then $[\psi(Y),\phi(X)]$ is a trace class operator and 
there exist unique $L^1(\mathbb{R})$- functions $\xi(t;\psi)$
and $\eta(\phi;\lambda)$ such that 
\begin{equation}\label{krehelhow1}
 -\im \Tr\{[\psi(Y),\phi(X)]\} = \int_a^b \phi '(t) \xi(t;\psi) dt
 =\int_a^b \psi '(\lambda) \eta(\phi;\lambda) d\lambda.
\end{equation}
Furthermore, if $\phi,\psi\in C_{1+}^1([a,b])$ or $\mathcal{P}_+([a,b])$, 
then $\xi(t;\psi)$, $\eta(\phi;\lambda)\geq 0$
for almost all 
$t,\lambda \in [a,b],$
\[
\int_a^b |\xi(t;\psi)|dt \leq \left\|-\im [\psi(Y),X]\right\|_1~~,~~
\int_a^b \xi(t;\psi) dt = \Tr\left(\psi'(Y)D^2\right) \quad \text{and}
\]
\[
\int_a^b |\eta(\phi;\lambda)|d\lambda \leq \left\|-\im [Y,\phi(X)]\right\|_1~~,~~
\int_a^b \eta(\phi;\lambda)d\lambda =\Tr \left(\phi'(X)D^2\right).
\]
\end{thm}
\begin{proof}
 At first we assume $\phi,\psi$ are real valued. 
 Now let us consider the self-adjoint operators $H_0=X$ and $H=\Exp^{\im\psi(Y)}X\Exp^{-\im\psi(Y)}$.
 Then
\begin{equation}\label{useq1}
\begin{split}
 H-H_0
 = \int_0^1\frac{d}{ds}\left(\Exp^{\im s\psi(Y)}X\Exp^{-\im s\psi(Y)} \right)ds
 = \im \int_0^1 \Exp^{\im s\psi(Y)}[\psi(Y),X]\Exp^{-\im s\psi(Y)}ds \in \mathcal{B}_1(\mathcal{H}),
\end{split}
\end{equation}
by lemma~\ref{lma1}. On the other hand for $\psi, \phi\in C_1^1([a,b])$, a  computation similar to that in 
\eqref{useq80} yields that
\begin{equation}\label{useq13}
\begin{split}
  [\psi(Y),\phi(X)] = \int_{\mathbb{R}} \im \hat{\psi}(\alpha) d\alpha \int_0^{\alpha}
  \Exp^{\im(\alpha-\beta) Y}[Y,\phi(X)]\Exp^{\im\beta Y} d\beta \in \mathcal{B}_1(\mathcal{H}),
\end{split}
\end{equation}
since $\int_{\mathbb{R}} |\hat{\psi}(\alpha)||\alpha|d\alpha <\infty$ and since 
$[Y,\phi(X)]\in \mathcal{B}_1(\mathcal{H})$, by lemma~\ref{lma1}. Similarly for $\phi$, $\psi(t)= \sum\limits_{j=0}^n c_jt^j\in \mathcal{P}([a,b]),$
by repeating the same calculations as in \eqref{useq95} we conclude that
\begin{equation}\label{useq70}
 \begin{split}
    [\psi(Y),\phi(X)] = \sum_{j=0}^nc_j[Y^j,\phi(X)]
  =  \sum_{j=0}^nc_j \sum_{k=0}^{j-1}Y^{j-k-1}[Y,\phi(X)] Y^k  \in \mathcal{B}_1(\mathcal{H}),.
 \end{split}
\end{equation}
since 
$[Y,\phi(X)]\in \mathcal{B}_1(\mathcal{H})$, by lemma~\ref{lma1}.
Thus by applying proposition~\ref{prop1} for the above operators $H,H_0$ with the function
$\phi$, we conclude that there exists a unique function $\tilde{\xi}(t;\psi)$ $\in L^1(\mathbb{R})$ such that
$\phi(H)-\phi(H_0)$ is trace class and 
\begin{equation}\label{useq3}
 \Tr \{\phi(H)-\phi(H_0)\} = \int_a^b \phi '(t)\tilde{\xi}(t;\psi)dt.
\end{equation}
Furthermore from equation \eqref{useq1} we conclude  that $H-H_0 \leq 0$, since 
$\im [\psi(Y),X]\leq 0$ by lemma~\ref{lmma2} for $\psi\in C_{1+}^1([a,b])$ or $\mathcal{P}_+([a,b])$. Therefore
from Proposition~\ref{prop1} we also note that $\tilde{\xi}(t;\psi)\leq 0$ for almost all
$t\in [a,b]$.
Now if we compute the left hand side of \eqref{useq3}, we get
\begin{equation}\label{useq4}
\begin{split}
 &\Tr \{\phi(H)-\phi(H_0)\} = \Tr\{\phi (\Exp^{\im\psi(Y)}X\Exp^{-\im\psi(Y)}) - \phi(X)\}
 =  \Tr\{\Exp^{\im\psi(Y)}\phi(X)\Exp^{-\im\psi(Y)})-\phi(X)\}\\
 & = \im  \Tr \{ \int_0^1 \Exp^{\im s\psi(Y)}[\psi(Y),\phi(X)]\Exp^{-\im s\psi(Y)}ds \}
 = \im  \Tr \{[\psi(Y),\phi(X)]\},
\end{split}
\end{equation}
where for the second equality we have used functional calculus, for the third equality we have 
used equation \eqref{useq1} and for the last equality we have used the cyclicity of trace. Thus by
combining \eqref{useq3} and \eqref{useq4} we have
\begin{equation}\label{useq5}
 -\im  \Tr \{[\psi(Y),\phi(X)]\} = -\int_a^b \phi'(t)\tilde{\xi}(t;\psi)dt
 = \int_a^b \phi '(t)\xi(t;\psi)dt,
\end{equation}
where $\xi(t;\psi)\equiv-\tilde{\xi}(t;\psi)\geq 0$. 
Next if we consider the operators $H_0=Y$ 
and $H=\Exp^{\im\phi(X)}Y\Exp^{-\im\phi(X)}$.
The by a similar calculation we conclude that
\begin{equation}\label{useq16}
 H-H_0
 =-\im \int_0^1 \Exp^{\im s\phi(X)} [Y,\phi(X)] \Exp^{-\im s \phi(X)} ds \in \mathcal{B}_1(\mathcal{H}),
\end{equation}
since $[Y,\phi(X)]\in \mathcal{B}_1(\mathcal{H})$, by lemma~\ref{lma1}.
Therefore by applying proposition~\ref{prop1} for  the above operators $H,H_0$ with the function
$\psi$, we conclude that there exists a unique function $\eta(\phi;\lambda)$ $\in L^1(\mathbb{R})$ such that
$\psi(H)-\psi(H_0)$ is trace class and
\begin{equation}\label{useq6}
 \Tr \{\psi(H)-\psi(H_0)\} = \int_a^b \psi '(\lambda)\eta(\phi;\lambda)d\lambda.
\end{equation}
Moreover, form equation \eqref{useq16} we note that $H-H_0\geq 0$, since $-\im [Y,\phi(X)]\geq 0$
by lemma~\ref{lmma2} for $\phi\in C_{1+}^1([a,b])$ or $\mathcal{P}_+([a,b])$.  Therefore
from proposition~\ref{prop1} we also note that $\eta(\phi;\lambda)\geq 0$ for almost all
$\lambda \in [a,b]$.
As before if we compute the left hand side of \eqref{useq6} we get
\begin{equation}\label{useq7}
\begin{split}
 &\Tr \{\psi(H)-\psi(H_0)\} = \Tr\{\psi (\Exp^{\im\phi(X)}Y\Exp^{-\im\phi(X)}) - \psi(Y)\}
 =  \Tr\{\Exp^{\im\phi(X)}\psi(Y)\Exp^{-\im\phi(X)})-\psi(Y)\}\\
 & = \im  \Tr \{ \int_0^1 \Exp^{\im s\phi(X)}[\phi(X),\psi(Y)]\Exp^{-\im s\phi(X)}ds \}
 = \im  \Tr \{[\phi(X),\psi(Y)]\} = -\im  \Tr \{[\psi(Y),\phi(X)]\}.
\end{split}
\end{equation}
Thus by combining \eqref{useq6} and \eqref{useq7} we have
\begin{equation}\label{useq8}
 -\im  \Tr \{[\psi(Y),\phi(X)]\} = \int_a^b \psi '(\lambda)\eta(\phi,\lambda)d\lambda.
\end{equation}
Therefore the conclusion of the theorem follows from \eqref{useq5} and \eqref{useq8}
for real valued $\phi, \psi$. The same above conclusions can be achieved
for complex valued functions $\phi,\psi\in C_1^1([a,b])$ or $\mathcal{P}([a,b])$ by decomposing
\[
 \phi =\phi_1 + \im \phi_2 \quad \text{and} \quad \psi =\psi_1 +\im \psi_2,
\]
and by applying the conclusion of the
theorem for real valued functions $\phi_1,\phi_2,\psi_1,\psi_2$.
By equation \eqref{useq1} of proposition \ref{prop1} and lemma \ref{lma1}, it follows that
\[
\int_a^b |\xi(t;\psi)|dt \leq \left\|H_0-H\right\|_1\leq \left\|-\im [\psi(Y),X]\right\|_1 
\]
and
\[
\int_a^b \xi(t;\psi)dt = \Tr (H_0-H)=\Tr \left(-\im [\psi(Y),X]\right) = \Tr \left(\psi'(Y)D^2\right).
\]
The other results for $\eta$ follows similarly. This completes the proof.
\end{proof}
\begin{rmrk}
 It is clear from equation \eqref{krehelhow1} that both $\xi(t; \cdot )$ and $\eta(\cdot ;\lambda)$ depend
 linearly on $\psi'$ and $\phi'$ respectively and not on $\psi$ and $\phi$ themselves as the left-handside 
 in \eqref{krehelhow1} appears to. Therefore, to avoid confusion it is preferable to replace $\psi'$, $\phi'$
 by $\psi$ and $\phi$ respectively, demand that $\psi$, $\phi$ $\in \mathcal{P}([a,b])$, and
 consequently replace $\psi$, $\phi$ by their indefinite integrals $\mathcal{J}(\psi)$ and $\mathcal{J}(\phi)$ respectively.
 Thus the equation \eqref{krehelhow1} now reads: For $\psi, \phi $ $\in \mathcal{P}([a,b])$
 \begin{equation}\label{rmkeq}
  \Tr \{-\im \left[\mathcal{J}(\psi)(Y), \mathcal{J}(\phi)(X)\right]\} =
  \int_a^b \phi(t) \xi (t;\psi) dt = \int_a^b \psi(\lambda) \eta (\phi;\lambda)d\lambda,
 \end{equation}
 where we have retained the earlier notation $\xi(t;\psi)$ and $\eta(\phi;\lambda)$. Furthermore, for almost all $t,\lambda \in [a,b]$, the maps 
 \[
  \mathcal{P}([a,b])\ni \psi \longmapsto \xi(t;\psi)\in L^1(\mathbb{R}) \quad \text{and} \quad
  \mathcal{P}([a,b])\ni \phi \longmapsto \eta(\phi;\lambda)\in L^1(\mathbb{R})
 \]
 are positive linear maps. The next theorem gives $L^1$-estimates for $\xi(\cdot; \psi)$ and $\eta(\phi; \cdot)$ which
 allows one to extend these maps for all $\psi, \phi$ $\in C([a,b])$.
\end{rmrk}
\begin{thm}\label{thm1}
Assume $\textbf{(A)}$.

\noindent \textup{(i)} Then 
 \[
  \mathcal{P}([a,b])\times \mathcal{P}([a,b])\ni (\psi,\phi) \longmapsto \Tr\{-\im \left[\mathcal{J}(\psi)(Y),
  \mathcal{J}(\phi)(X)\right]\}
 \]
can be extended as a positive linear map on $C([a,b])\times C([a,b])$. Furthermore if $\Delta,$ $ \Omega$ $\in Borel ([a,b]),$
then 
\begin{equation}\label{useq73}
  \Tr\{-\im \left[\mathcal{J}(\chi_{_{\Delta}})(Y),
  \mathcal{J}(\chi_{_{\Omega}})(X)\right]\} = \int_{\Omega} \xi(t; \Delta)dt = \int_{\Delta} \eta (\Omega;\lambda)d\lambda,
\end{equation}
where we have written $\xi(t;\Delta)$ for $\xi(t; \chi_{_{\Delta}})$ and $\eta(\Omega;\lambda)$ for 
$\eta(\chi_{_{\Omega}}; \lambda)$. For almost all fixed $t,\lambda \in [a,b]$, $\xi(t; \cdot)$ and $\eta(\cdot; \lambda)$
are countably additive positive measures such that 
\[
 \int_a^b \xi(t;\Delta) dt = \Tr \left(\chi_{_{\Delta}}(Y)D^2\right), \int_a^b \eta(\Omega; \lambda) d\lambda = 
 \Tr \left(\chi_{_{\Omega}}(X)D^2\right),
\]
and
\[
 \int_a^b \xi(t; [a,b]) dt = \int_a^b \eta([a,b];\lambda)d\lambda = \Tr (D^2).
\]
\noindent \textup{(ii)} The set functions 
\[
 Borel([a,b])\ni \Delta \longmapsto \xi (t: \Delta) \quad \text{and} \quad Borel([a,b])\ni \Omega \longmapsto \eta(\Omega; \lambda)
\]
are absolutely continuous with respect to the Lebesgue measures and the Radon-Nikodym derivatives satisfy:
\[
 \frac{\xi(t; d\lambda)}{d\lambda} = \frac{\eta(dt;\lambda)}{dt}\equiv r(t,\lambda) \geq 0
\]
for almost all $t,\lambda,$ with $\|r\|_{L^1([a,b]^2)} = \Tr (D^2)$.
\vspace{0.1in}

\noindent \textup{(iii)} The statement of the theorem \ref{thm3} now takes the form: For $\psi, \phi \in C^1([a,b])$ 
\begin{equation}\label{useq53}
 \Tr\{-\im [\psi(Y),\phi(X)]\} = \int_{[a,b]^2} \phi'(t) \psi'(\lambda) r(t,\lambda) dt d\lambda,
\end{equation}
with the unique non-negative $L^1([a,b]^2)$ function $r$, which is sometimes called Carey-Pincus Principal function.
\end{thm}

\begin{proof}
 Let $\psi, \phi \in \mathcal{P}([a,b])$, then $\mathcal{J}(\psi)$ and $\mathcal{J}(\phi)$ are also polynomials.
As in \eqref{useq70}, a similar computation with $\psi, \phi \in \mathcal{P}([a,b])$ and if 
$\mathcal{J}(\phi)(t) = \sum\limits_{j=0}^n c_jt^j$ leads to 
\[
 -\im \left[\mathcal{J}(\psi)(Y),\mathcal{J}(\phi)(X)\right]  =
 -\im  \sum_{j=0}^n c_j \left(\sum_{k=0}^{j-1}X^k \left[\mathcal{J}(\psi)(Y),X\right]X^{j-k-1}\right)
\]
and taking trace 
\begin{equation}\label{useq71}
 \begin{split}
 & \Tr\{-\im \left[\mathcal{J}(\psi)(Y),\mathcal{J}(\phi)(X)\right]\}
  = -\im \Tr\{\sum_{j=1}^n j c_j X^{j-1}\left(\left[\mathcal{J}(\psi)(Y),X\right]\right)\}\\
 &\hspace{5.05cm} = \Tr\{\phi(X) \left(-\im \left[\mathcal{J}(\psi)(Y),X\right]\right)\}
 \end{split}
\end{equation}
and interchanging the role of $X$ and $Y$ (along with an associated negative sign) the above is equal to
\begin{equation}\label{useq50}
 \Tr\{-\im \left[\mathcal{J}(\psi)(Y),\mathcal{J}(\phi)(X)\right]\}
 = \Tr\{\psi(Y)\left(-\im\left[Y,\mathcal{J}(\phi)(X)\right]\right)\},
\end{equation}
and all these expressions are also equal to (by theorem \ref{thm3}) 
\[
 \int_a^b \phi(t) \xi(t; \psi) dt = \int_a^b \psi(\lambda) \eta(\phi; \lambda) d\lambda
\]
for respective $\phi$ and $\psi$. Now let $\phi =\phi_+ - \phi_-$ and $\psi = \psi_+-\psi_-$, then 
$\phi_{\pm}$, $\psi_{\pm}$ are all non-negative. The domains of definitions of $\phi_{\pm}$ are open sets
which are each a disjoint union of a countable collection of open intervals and furthermore, clearly 
\emph{Supp} $\phi_{+}$ $\cap$ \emph{Supp} $\phi_-$ $=$ $ \{t \in [a,b]| \phi(t) =0\}$, which is a finite discrete set.
Therefore $\phi_+$ and $\phi_-$ and hence $|\phi|=\phi_+ + \phi_-$ are polynomials if $\phi \in \mathcal{P}([a,b])$.
By lemma~\ref{lmma6}, $-\im [\psi(Y),\phi(X)]\in \mathcal{B}_{1\pm}(\mathcal{H})$
according as $\psi, \phi \in \mathcal{P}_{\pm}([a,b])$ respectively.
Therefore by linearity, 
\begin{equation*}
 \begin{split}
 & -\im \left[\mathcal{J}(\psi)(Y),\mathcal{J}(\phi)(X)\right]
  = -\im \left[\mathcal{J}(\psi_+)(Y),\mathcal{J}(\phi_+)(X)\right]
  + \im \left[\mathcal{J}(\psi_-)(Y),\mathcal{J}(\phi_+)(X)\right]\\
  & \hspace{4.8cm} +  \im \left[\mathcal{J}(\psi_+)(Y),\mathcal{J}(\phi_-)(X)\right]
  -  \im \left[\mathcal{J}(\psi_-)(Y),\mathcal{J}(\phi_-)(X)\right],
 \end{split}
\end{equation*}
which by using lemma~\ref{lmma6} and equations \eqref{useq71} and \eqref{useq50} we conclude that
\begin{equation}
 \begin{split}
 & \left\|-\im \left[\mathcal{J}(\psi)(Y),\mathcal{J}(\phi)(X)\right]\right\|_1\\
  & \leq \Tr\{-\im \left[\mathcal{J}(\psi_+)(Y),\mathcal{J}(\phi_+)(X)\right]\}
  + \Tr\{-\im \left[\mathcal{J}(\psi_-)(Y),\mathcal{J}(\phi_+)(X)\right]\} \\
  &  +  \Tr\{-\im \left[\mathcal{J}(\psi_+)(Y),\mathcal{J}(\phi_-)(X)\right]\}
  + \Tr\{-  \im \left[\mathcal{J}(\psi_-)(Y),\mathcal{J}(\phi_-)(X)\right]\}\\
 & = \Tr \{|\phi|(X) \left(-\im \left[\mathcal{J}(|\psi|)(Y),X\right]\right)\}
 \end{split}
\end{equation}
and similarly 
\begin{equation}
  \left\|-\im \left[\mathcal{J}(\psi)(Y),\mathcal{J}(\phi)(X)\right]\right\|_1
  \leq \Tr\{|\psi|(Y)\left(-\im \left[Y,\mathcal{J}(|\phi|)(X)\right]\right)\}.
\end{equation}
These estimates allow us to extend the formulae to $\psi$, $\phi$ $\in C([a,b])$. Indeed, 
consider a sequence $\{\phi_m\}$ of polynomials converging to $\phi \in C([a,b])$ uniformly
in $[a,b]$. Then for a fixed $\psi \in \mathcal{P}([a,b])$ 
\[
 \left\|-\im \left[\mathcal{J}(\psi)(Y),\mathcal{J}(\phi_n-\phi_m)(X)\right]\right\|_1
 \leq \Tr\{|\phi_n-\phi_m|(X) \left(-\im \left[\mathcal{J}(|\psi|)(Y),X\right]\right)\}
\]
which converges to zero as $n,m\longrightarrow \infty$, since 
$\||\phi_n-\phi_m|(X)\| = \|\phi_n-\phi_m\|_{\infty}$ and since $-\im \left[\mathcal{J}(|\psi|)(Y),X\right]
\in \mathcal{B}_1(\mathcal{H})$. Proceeding similarly for $\psi$ with $\phi$ fixed, one concludes that for $\psi$, $\phi$ $\in C([a,b])$
\[
-\im \left[\mathcal{J}(\psi)(Y),\mathcal{J}(\phi)(X)\right] \in \mathcal{B}_1(\mathcal{H}) \quad \text{and}
\]
\begin{equation}
 \begin{split}
& \Tr\{-\im \left[\mathcal{J}(\psi)(Y),\mathcal{J}(\phi)(X)\right]\} =
\Tr\{\phi(X) \left(-\im \left[\mathcal{J}(\psi)(Y),X\right]\right)\}\\
&\hspace{5.05cm}  = \Tr\{\psi(Y)\left(-\im \left[Y,\mathcal{J}(\phi)(X)\right]\right)\}.
 \end{split}
\end{equation}
On the other hand 
\begin{equation}
 \begin{split}
&\left|\int_a^b [\phi_n(t)-\phi(t)] \xi(t;\psi) dt)\right|
\leq \|\phi_n-\phi\|_{\infty} \int_a^b |\xi(t;\psi)|dt \\
& \leq \|\phi_n-\phi\|_{\infty} \left\|-\im \left[\mathcal{J}(\psi)(Y),X\right]\right\|_1,
 \end{split}
\end{equation}
and similarly 
\begin{equation*}
 \begin{split}
&\left|\int_a^b \psi(\lambda) \eta(\phi_n-\phi; \lambda) d\lambda\right| 
\leq \|\psi\|_{\infty} \int_a^b |\eta(\phi_n-\phi; \lambda)| d\lambda\\
& \leq \|\psi\|_{\infty} \left\|-\im \left[Y, \mathcal{J}(\phi_n-\phi)(X)\right]\right\|_1
\leq \|\psi\|_{\infty} \Tr\{-\im \left[Y, \mathcal{J}(|\phi_n-\phi|)(X)\right]\},
 \end{split}
\end{equation*}
which by lemma \ref{lma1} is equal to
\[
 \|\psi\|_{\infty} \Tr\{|\phi_n-\phi|(X) D^2\} \longrightarrow 0 \quad \text{as} \quad n\longrightarrow \infty.
\]
Therefore, one has that all the equal expressions in \eqref{useq71} and \eqref{useq50} are also equal to
\begin{equation}\label{useq72}
 \int_a^b \phi(t) \xi(t;\psi) dt = \int_a^b \psi(\lambda) \eta(\phi;\lambda) d\lambda
\end{equation}
for all $\phi\in C([a,b])$ and $\psi \in \mathcal{P}([a,b])$. Now holding $\phi$ fixed in $C([a,b])$ and
approximating $\psi\in C([a,b])$ by a sequence $\{\psi_m\}$ of polynomials, exactly as above, one establishes the equality 
\eqref{rmkeq} for all $\phi, \psi \in C([a,b])$.
\vspace{0.1in}

\noindent Next we want to approximate characteristic functions by continuous functions in the same formula. Let
$\{\phi_m\}$ be a sequence of uniformly bounded continuous 
functions in $[a,b]$ converging point-wise almost everywhere to $\chi_{_{\Omega}}$ with 
$\Omega \in Borel (\sigma(X))$ and let $\{\psi_m\}$ be a sequence similarly approximating $\chi_{_{\Delta}}$
for $\Delta \in Borel(\sigma(Y))$. Then by \eqref{useq72}, we have that for $\psi \in C([a,b])$
\begin{equation*}
 \begin{split}
& \int_a^b \left[\phi_n(t) - \chi_{_{\Omega}} \right]  \xi(t;\psi) dt 
= \Tr\{-\im \left[\mathcal{J}(\psi)(Y),\mathcal{J}(\phi_n-\chi_{_{\Omega}})(X)\right]\}\\
& \hspace{4.4cm} =\Tr\{(\phi_n-\chi_{_{\Omega}})(X)\left( \left[\mathcal{J}(\psi)(Y),X\right]\right)\},
 \end{split}
\end{equation*}
which converges to zero as $n\longrightarrow \infty$ since $\phi_n(X)$ converges strongly
to $\chi_{_{\Omega}}(X)$ and since $-\im \left[\mathcal{J}(\psi)(Y),X\right]
\in \mathcal{B}_1(\mathcal{H})$. Thus
\begin{equation}
 \begin{split}
\int_{\Omega}  \xi(t;\psi) dt  = \lim_{n\longrightarrow \infty} \int_a^b \phi_n(t)  \xi(t;\psi) dt
= \lim_{n\longrightarrow \infty} \int_a^b \psi(\lambda) \eta(\phi_n;\lambda) d\lambda.
 \end{split}
\end{equation}
Now 
\begin{equation*}
\begin{split}
 \int_a^b |\eta(\phi_n-\chi_{_{\Omega}});\lambda| d\lambda \leq 
 \left\|-\im \left[Y,\mathcal{J}(\phi_n-\chi_{_{\Omega}})(X)\right]\right\|_1
 \leq \Tr\{(|\phi_n-\chi_{_{\Omega}})|)(X) D^2\},
\end{split}
\end{equation*}
which converges to zero as $n\longrightarrow \infty$ by the same reasoning exactly as above, and therefore 
\begin{equation}
 \lim_{n\longrightarrow \infty} \int_a^b \psi(\lambda) \eta (\phi_n;\lambda) d\lambda = 
 \int_a^b \psi(\lambda) \eta(\chi_{_{\Omega}};\lambda) d\lambda.
\end{equation}
In the equality 
\begin{equation*}
\int_{\Omega} \xi(t;\psi_m) dt = \int_a^b \psi_m(\lambda) \eta (\chi_{_{\Omega}}; \lambda) d\lambda 
\end{equation*}
with $\{\psi_m\}$ approximating $\chi_{_{\Delta}}$ as described earlier, we get required equality \eqref{useq73}. 
This completes the proof of $\textup(i)$.
\vspace{0.1in}

\noindent \textup{(ii)} From the equality, for $\Omega \in Borel (\sigma(X))$,
$\Delta \in Borel (\sigma(Y)),$ 
\[
 \int_{\Omega} \xi(t;\Delta) dt = \int_{\Delta} \eta(\Omega; \lambda) d\lambda
\]
with $\xi$ and $\eta$ both non-negative, it follows that
\[
 \textup{Lebesgue measure}(\Omega) =0 \Longrightarrow \eta(\Omega;\lambda) = 0 \quad \text{and similarly}
\]
\[
 \textup{Lebesgue measure}(\Delta) =0 \Longrightarrow \xi(t;\Delta) = 0, \quad \text{and clearly }
\]
for fixed $t,\lambda$; $\xi(t;\cdot)$ and $\eta(\cdot; \lambda)$ are countably additive point-wise set 
functions. Therefore they are both absolutely continuous with respect to the respective Lebesgue measures, and we set
\[
 r(t,\lambda) = \frac{\xi(t;d\lambda)}{d\lambda} = \frac{\eta(dt;\lambda)}{dt} \geq 0.
\]
The uniqueness of $r$ follows from the equation \eqref{useq53}  and the fact that $r\in L^1([a,b]^2)$ and that it has compact support.
This completes the proof.
\end{proof}
\noindent Next, we want to compute the trace of 
\begin{equation*}
 \Tr\{[\alpha(X)\psi(Y),\phi(X)]\} = \Tr\{\alpha(X)[\psi(Y),\phi(X)]\} 
\end{equation*}
and by symmetry between $X$ and $Y$,~$
 \Tr\{[\alpha(X)\psi(Y),\beta(Y)]\} = -\Tr\{\psi(Y)[\beta(Y),\alpha(X)]\},$
where $\phi, \psi, \alpha, \beta \in \mathcal{P}([a,b]),$ which 
constitutes the next theorem.
\begin{thm}\label{thm2}
Let $T$ satisfy $\textbf{(A)}$.
Let $\phi,\psi,\alpha,\beta \in \mathcal{P}([a,b]).$
Then $[\alpha(X)\psi(Y),\phi(X)] \in \mathcal{B}_1(\mathcal{H})$ and 
\begin{equation}
\begin{split}
 & -\im \Tr\{[\alpha(X)\psi(Y),\phi(X)]\} =
  \int_{[a,b]^2}\alpha(t)\phi'(t)\psi'(\lambda)r(t,\lambda)dtd\lambda\\
 & =\int_{[a,b]^2} -J(\alpha \psi, \phi)(t,\lambda)r(t,\lambda)dtd\lambda,
\end{split}
 \end{equation}
where $r$ is the function obtained in theorem~\ref{thm1} and 
\[J(\alpha \psi, \phi)(t,\lambda)= 
\frac{\partial}{\partial t}(\alpha(t) \psi(\lambda))\frac{\partial}{\partial \lambda}(\phi(t))
-\frac{\partial}{\partial \lambda}(\alpha(t) \psi(\lambda))\frac{\partial}{\partial t}(\phi(t))\]
is the Jacobian of  $\alpha \psi$ and $\phi $ in $[a,b]\times [a,b]\equiv [a,b]^2$.
Similarly, $[\alpha(X)\psi(Y),\beta(Y)]\in \mathcal{B}_1(\mathcal{H})$ and
\begin{equation}
 \begin{split}
& -\im \Tr\{[\alpha(X)\psi(Y),\beta(Y)]\}=
 \int_{[a,b]^2}-\alpha'(t)\psi(\lambda)\beta'(\lambda)r(t,\lambda)dtd\lambda \\
& =\int_{[a,b]^2} -J(\alpha \psi, \beta)(t,\lambda)r(t,\lambda)dtd\lambda,
\end{split}
\end{equation}
where $r$ is the function obtained in theorem~\ref{thm1} and 
\[J(\alpha \psi, \phi)(t,\lambda)= 
\frac{\partial}{\partial t}(\alpha(t) \psi(\lambda))\frac{\partial}{\partial \lambda}(\beta(\lambda))
-\frac{\partial}{\partial \lambda}(\alpha(t) \psi(\lambda))\frac{\partial}{\partial t}(\beta(\lambda))\]
is the Jacobian of  $\alpha \psi$ and $\beta $ in $[a,b]^2$.
\end{thm}
\begin{proof}
Using \eqref{useq13} we say that 
\[
 [\alpha(X)\psi(Y),\phi(X)] = \alpha(X)[\psi(Y),\phi(X)]\in \mathcal{B}_1(\mathcal{H}).
\]
Next from \eqref{useq70} we conclude for $\psi,\phi\in \mathcal{P}([a,b])$ that 
\begin{equation}\label{useq27}
 \begin{split}
  -\im \Tr\{[\psi(Y),\phi(X)]\} = -\im \Tr \{\phi'(X) [\psi(Y),X]\}
  = \int_a^b \phi'(t) \Tr\left(E^{(X)}(dt)\{-\im [\psi(Y),X]\}\right),
 \end{split}
\end{equation}
where we have used spectral theorem for the self-adjoint operator $X$ and 
$E^{(X)} (.)$ is the spectral family of  $X$. On the other hand from theorem~\ref{thm1}$(iii)$ we conclude that
\begin{equation}\label{useq28}
 -\im \Tr\{[\psi(Y),\phi(X)]\} = \int_{[a,b]^2}
 \phi'(t)\psi'(\lambda)r(t,\lambda) dt d\lambda,
\end{equation}
for $\psi,\phi\in \mathcal{P}([a,b])$. Therefore combining \eqref{useq27} and \eqref{useq28}
we get
\begin{equation}\label{useq29}
 \int_a^b \phi'(t) \Tr\left(E^{(X)}(dt)\{-\im [\psi(Y),X]\}\right)
 = \int_a^b \phi'(t) \left(\int_a^b \psi'(\lambda) r(t,\lambda) d\lambda\right) dt, 
\end{equation}
for $\psi,\phi\in \mathcal{P}([a,b])$. Since 
\[
 \Delta \longrightarrow \Tr\left(E^{(X)}(\Delta)\{-\im [\psi(Y),X]\}\right)\quad \
 (\Delta \subseteq \mathbb{R},~ \text{a Borel subset of}~ \mathbb{R})
\]
is a complex measure with finite total variation and $r\in L^1{[a,b]^2}$, then we can extend 
the above equality \eqref{useq29} to all $\phi\in C([a,b])$ and therefore
\begin{equation}\label{useq30}
 \begin{split}
 & \int_a^b \phi(t) \Tr\left(E^{(X)}(dt)\{-\im [\psi(Y),X]\}\right)\\
 & \hspace{1.5in} 
 = \int_a^b \phi(t) \left(\int_a^b \psi'(\lambda) r(t,\lambda) d\lambda\right) dt
 \quad \text{for all}\quad
 \phi\in C([a,b]).
\end{split}
 \end{equation}
Thus by approximating characteristic function $\chi_{_{\Delta}}$ (for  Borel subset 
$\Delta \subseteq \mathbb{R}$) through continuous functions we conclude from
the above equation \eqref{useq30}
that
\begin{equation}
\Tr\left(E^{(X)}(\Delta)\{-\im [\psi(Y),X]\}\right)
 = \int_{t\in \Delta} dt \left(\int_a^b \psi'(\lambda) r(t,\lambda) d\lambda\right), 
\end{equation}
which shows that the measure
\[
  \Delta \longrightarrow \Tr\left(E^{(X)}(\Delta)\{-\im [\psi(Y),X]\}\right)
\]
is absolutely continuous with respect to a Lebesgue measure $dt$ and
\begin{equation}\label{useq31}
 \Tr\left(E^{(X)}(dt)\{-\im [\psi(Y),X]\}\right)
 = \left(\int_a^b \psi'(\lambda) r(t,\lambda) d\lambda\right) dt. 
\end{equation}
As in \eqref{useq71}, a similar computation with $\psi, \phi \in \mathcal{P}([a,b])$ and if
$\phi(\lambda) = \sum\limits_{j=0}^nb_j\lambda^j$ leads to
\begin{equation}\label{useq32}
 \begin{split}
& -\im \Tr\{[\alpha(X)\psi(Y),\phi(X)]\} = -\im \Tr\{\alpha(X)[\psi(Y),\phi(X)]\}
= -\im \Tr\{\alpha(X)\sum_{j=0}^nb_j[\psi(Y),X^j]\}\\
& = -\im \sum_{j=0}^nb_j \sum_{k=0}^{j-1}\Tr\{\alpha(X) X^k[\psi(Y),X]X^{j-k-1}\}
=-\im \sum_{j=0}^nb_j \sum_{k=0}^{j-1}\Tr\{\alpha(X) X^{j-1}[\psi(Y),X]\}\\
& = -\im \Tr\{ \alpha(X) \sum_{j=1}^n j b_j X^{j-1}[\psi(Y),X]\}
= -\im \Tr\{ \alpha(X) \phi'(X)[\psi(Y),X]\}\\
& = -\im \Tr\{\int_a^b \alpha(t) \phi'(t) E^{(X)}(dt)[\psi(Y),X]\}
= \int_a^b \alpha(t) \phi'(t) \Tr\left(E^{(X)}(dt)\{-\im [\psi(Y),X]\}\right),
 \end{split}
\end{equation}
where for fifth equality we have used the cyclicity of trace and for eighth equality we 
have used the spectral theorem for the self-adjoint operator $X$ and 
$E^{(X)} (.)$ is the spectral family of  $X$. Next by combining \eqref{useq31}
and \eqref{useq32} we conclude that 
\begin{equation*}
 \begin{split}
  & -\im \Tr\{[\alpha(X)\psi(Y),\phi(X)]\} = 
   \int_a^b \alpha(t) \phi'(t) \left(\int_a^b \psi'(\lambda) r(t,\lambda) d\lambda\right) dt\\
   & =\int_{[a,b]^2}\alpha(t)\phi'(t)\psi'(\lambda)r(t,\lambda)dtd\lambda
   =\int_{[a,b]^2} -J(\alpha \psi, \phi)(t,\lambda)r(t,\lambda)dtd\lambda.
 \end{split}
\end{equation*}
Next by interchanging the role of $X$ and $Y$ in the above calculations, we get that
\begin{equation*}
 \begin{split}
& -\im \Tr\{[\alpha(X)\psi(Y),\beta(Y)]\}=
 \int_{[a,b]^2}-\alpha'(t)\psi(\lambda)\beta'(\lambda)r(t,\lambda)dtd\lambda \\
& =\int_{[a,b]^2} -J(\alpha \psi, \beta)(t,\lambda)r(t,\lambda)dtd\lambda.
\end{split}
\end{equation*}
This completes the proof.
\end{proof}
\begin{rmrk}
 If $T$ satisfy $\textbf{(A)}$, then the conclusion of the above theorem \ref{thm2}
 also can be achieved for $\phi,\psi,\alpha,\beta \in C_1^1([a,b]).$
\end{rmrk}
\noindent The next theorem replaces effectively the so called 
\textquotedblleft Wallach's Collapse Theorem"
\cite{martinputinar}.
\begin{thm}\label{Wthm}
Let $T$ be as in the statement of theorem~\ref{thm2}.
Let $\phi,\psi, \alpha,\beta \in \mathcal{P}([a,b])$. 
Then the following is true
\begin{equation*}
 -\im \Tr \{[\alpha(X) \psi(Y),\phi(X) \beta(Y)]\}
 = \int_{[a,b]^2} -J(\alpha \psi, \phi \beta)(t,\lambda) r(t,\lambda) dt d\lambda,
\end{equation*}
where 
\[
J(\alpha \psi, \phi \beta)(t,\lambda) = 
\frac{\partial}{\partial t}(\alpha(t) \psi(\lambda))
\frac{\partial}{\partial \lambda}(\phi(t) \beta(\lambda))
-\frac{\partial}{\partial \lambda}(\alpha(t) \psi(\lambda))
\frac{\partial}{\partial t}(\phi(t) \beta(\lambda))
\]
is the Jacobian of  $\alpha \psi$ and $\phi \beta$ in $[a,b]^2$.
\end{thm}
\begin{proof}
By simple computation we get the following
\begin{equation*}
\begin{split}
& [\alpha(X) \psi(Y),\phi(X) \beta(Y)]- \{
[\alpha(X)(\psi \beta)(Y), \phi(X)]+[(\alpha \phi)(X)\psi(Y),\beta(Y)]\}\\
& = \alpha(X) \psi(Y) \phi(X) \beta(Y) - \phi(X) \beta(Y) \alpha(X) \psi(Y) 
- \alpha(X) \psi(Y) \beta(Y) \phi(X) \\
& + \beta(Y) \alpha(X) \phi(X) \psi(Y)\\
& = \alpha(X) \psi(Y) [\phi(X), \beta(Y) ] - [\phi(X), \beta(Y) ] \alpha(X) \psi(Y)
\in \mathcal{B}_1(\mathcal{H}),
\end{split}
\end{equation*}
by equation \eqref{useq13} and therefore  
\begin{equation}
\begin{split} 
& \Tr\{[\alpha(X) \psi(Y),\phi(X) \beta(Y)]- (
[\alpha(X)(\psi \beta)(Y), \phi(X)]+[(\alpha \phi)(X)\psi(Y),\beta(Y)]) \}\\
& = \Tr \{\alpha(X) \psi(Y) [\phi(X), \beta(Y) ] - [\phi(X), \beta(Y) ] \alpha(X) \psi(Y) \} = 0,
\end{split}
\end{equation}
where we have used the cyclicity of trace and the fact that 
$[\phi(X), \beta(Y) ] \in \mathcal{B}_1(\mathcal{H})$.
Thus we have shown that 
\begin{equation}\label{useq9}
 -\im \Tr\{[\alpha(X) \psi(Y),\phi(X) \beta(Y)]\}
 = -\im \Tr\{[\alpha(X)(\psi \beta)(Y), \phi(X)]\} -\im \Tr\{ [(\alpha \phi)(X)\psi(Y),\beta(Y)]\}.
\end{equation}
Therefore by using theorem~\ref{thm2} we compute the right hand side of \eqref{useq9} to get
\begin{equation}\label{useq10}
\begin{split}  
 &-\im \Tr\{[\alpha(X)(\psi \beta)(Y), \phi(X)]\} -\im \Tr\{ [(\alpha \phi)(X)\psi(Y),\beta(Y)]\}\\
 & = \int_{[a,b]^2} \alpha(t) \phi'(t)(\psi \beta)'(\lambda) r(t,\lambda)dt d\lambda
 - \int_{[a,b]^2} (\alpha \phi)'(t) \psi(\lambda)\beta'(\lambda) r(t,\lambda)dt d\lambda\\
 & = \int_{[a,b]^2} -J(\alpha \psi, \phi \beta)(t,\lambda) r(t,\lambda) dt d\lambda.
\end{split}
\end{equation} 
Therefore combining \eqref{useq9} and \eqref{useq10} we get
\begin{equation*}
 -\im \Tr\{[\alpha(X) \psi(Y),\phi(X) \beta(Y)]\}
 = \int_{[a,b]^2} -J(\alpha \psi, \phi \beta)(t,\lambda) r(t,\lambda) dt d\lambda.
\end{equation*}
This completes the proof.
\end{proof}
Now we are in a position to state our main result, the Helton-Howe-Carey-Pincus trace formula \cite{alexpeller, clancybook, heltonhowe73,heltonhowe75,martinputinar}.
\begin{thm}
 Let $\Psi(t,\lambda) = \sum\limits_{j=1}^n c_j \alpha_j(t) \psi_j(\lambda)$
 and $\Phi(t,\lambda) = \sum\limits_{k=1}^m d_j \phi_k(t) \beta_k(\lambda),$
 ($m, n$ $\in \mathbb{N}$) and $\alpha_j,$ $\psi_j,$ $\phi_j,$ $\beta_j$ are
 all in $\mathcal{P}([a,b])$. Then $-\im \left[\Psi(X,Y), \Phi(X,Y)\right]$
 $\in \mathcal{B}_1(\mathcal{H})$ and 
 \[
  \Tr\{-\im \left[\Psi(X,Y), \Phi(X,Y)\right]\}
  = \int_{[a,b]^2} J(\Psi,\Phi)(t,\lambda) r(t,\lambda) dt d\lambda.
 \]

\end{thm}
\begin{proof}
Proof follows easily by applying theorem~\ref{Wthm} and the fact that
\[
 \Tr\{-\im \left[\Psi(X,Y), \Phi(X,Y)\right]\}
 = \sum\limits_{j=1}^n \sum\limits_{k=1}^m c_jd_k
 \Tr\{-\im \left[\alpha_j(X) \psi_j(Y), \phi_k(X) \beta_k(Y)\right]\}.
\]
\end{proof}

\noindent\textbf{Acknowledgement:} The first author is grateful to ISI, Bangalore Centre and IIT, Guwahati for 
warm hospitality and the second author
is grateful to Jawaharlal Nehru Centre for Advanced Scientific Research
and SERB-Distinguished Fellowship for support.


\end{document}